\documentclass[12pt]{amsart}

\usepackage{amssymb}
\usepackage{hyperref}
\usepackage{tikz}

\usetikzlibrary{positioning}

\newtheorem{thm}{Theorem}[section]
\newtheorem{lemma}[thm]{Lemma}
\newtheorem{prop}[thm]{Proposition}
\newtheorem{cor}[thm]{Corollary}

\newtheorem*{thm:thm1}{Theorem \ref{thm1}}
\newtheorem*{thm:thm2}{Proposition \ref{thm2}}
\newtheorem*{thm:thm3}{Theorem \ref{thm3}}
\newtheorem*{cor:cor1}{Corollary \ref{cor1}}

\theoremstyle{definition}

\newtheorem{rem}[thm]{Remark}
\newtheorem{question}[thm]{Question}

\theoremstyle{remark}

\newcommand{\Homeo}{\mathop{\rm Homeo}}
\newcommand{\Diff}{\mathop{\rm Diff}}
\newcommand{\id}{\mathop{\rm id}}
\newcommand{\rot}{\mathop{\rm rot}}

\begin{document}

\author{Michael P. Cohen}
\address{Michael P. Cohen,
Department of Mathematics and Statistics,
Carleton College,
One North College Street,
Northfield, MN 55057}
\email{mcohen@carleton.edu}

\thanks{\textit{Acknowledgements.}  The author thanks E. Militon for suggesting a result which eventually became Theorem \ref{thm1}, and for improving the example in Proposition \ref{thm2}.  The author also thanks Christian Rosendal for helpful suggestions.}



\title{Maximal pseudometrics and distortion of circle diffeomorphisms}

\begin{abstract}  We initiate a study of distortion elements in the Polish groups $\Diff_+^k(\mathbb{S}^1)$ ($1\leq k<\infty$), as well as $\Diff_+^{1+AC}(\mathbb{S}^1)$, in terms of maximal metrics on these groups.  We classify distortion in the $k=1$ case: a $C^1$ circle diffeomorphism is $C^1$-undistorted if and only if it has a hyperbolic periodic point.  On the other hand, answering a question of Navas, we exhibit analytic circle diffeomorphisms with only non-hyperbolic fixed points which are $C^{1+AC}$-undistorted, and hence $C^k$-undistorted for all $k\geq 2$.  In the appendix, we exhibit a maximal metric on $\Diff_+^{1+AC}(\mathbb{S}^1)$, and observe that this group is quasi-isometric to a hyperplane of $L^1(I)$.
\end{abstract}

\maketitle

\section{Introduction}

A pseudometric $d$ on a topological group $G$ is called \textit{maximal} if it is continuous, right-invariant, and for every continuous right-invariant pseudometric $\rho$ on $G$, there exist constants $K$ and $C$ such that\\

\begin{center} $\rho \leq K\cdot d+C$.
\end{center}
\vspace{.3cm}

By definition any two maximal pseudometrics on $G$ are quasi-isometrically equivalent, and thus if a maximal pseudometric $d$ exists, then it is a representative of a canonical \textit{quasi-isometry type} for $G$: the quasi-isometric equivalence class of all maximal pseudometrics on $G$.  Moreover, it turns out that this equivalence class contains all right-invariant Cayley metrics on $G$ which are induced by an open symmetric coarsely bounded generating set $\Sigma\subseteq G$.  (A set in $G$ is called coarsely bounded if it is bounded with respect to every continuous right-invariant pseudometric on $G$.)  This very general perspective for the study of the large scale geometry of Polish groups has been advanced and thoroughly developed by Rosendal  \cite{rosendal_2018a}.

Let $f\in G$ and let $d$ be a maximal pseudometric on $G$.  Adopting the definition proposed by Mann and Rosendal in \cite{mann_rosendal_2018a}, we will say that $f$ is \textit{distorted} if the inclusion map $\langle f\rangle\hookrightarrow G$ is not a quasi-isometric embedding, with respect to any right-invariant Cayley metric on $\langle f\rangle$ and with respect to $d$ on $G$; otherwise $f$ is \textit{undistorted}.  This definition does not depend on the choice of maximal pseudometric $d$.

By definition every torsion element of $G$ is undistorted.  To understand distortion of non-torsion elements $f\in G$, we study the limit $L=\displaystyle\lim_{n\rightarrow\infty}\dfrac{d(f^n,e)}{n}$.  By the right-invariance of $d$ the sequence $(d(f^n,e))_{n=0}^\infty$ is subadditive, and hence the limit exists.  A non-torsion element $f\in G$ is distorted if and only if $L=0$, and undistorted if and only if $L>0$.

The study of \textit{distortion}, as an informal notion, has been essential in the study of the dynamics of homeomorphisms and diffeomorphisms of compact manifolds (see for example \cite{demelo_vanstrien_1993a}, \cite{franks_handel_2006a}, \cite{calegari_freedman_2006a}, \cite{militon_2014a}, \cite{navas_2018a})\textemdash but the specific definitions that authors have employed have depended on context and purpose, and there does not appear to be a universally accepted definition of the term.  Below, we emphasize some of the key features of the definition we adopt in this paper:

\begin{enumerate}
		\item Distortion \textit{a priori} is a topological group property, i.e. the quasi-isometry type of the group, and its set of distortion elements, are determined by both the group structure and the underlying topology on the group.
		\item Conjugation induces a topological group automorphism.  Therefore the property of being distorted (or undistorted) is a natural conjugacy invariant of an element $f\in G$.
		\item If $\ell$ is any continuous length function on $G$, then $\ell$ induces a continuous right-invariant metric defined by $d_\ell(x,y)=\ell(xy^{-1})$, and we have $d_\ell\leq K\cdot d+C$ for some constants $K$ and $C$, for any maximal pseudometric $d$.  Therefore if $f\in G$ is distorted, then $\ell(f^n)/n\rightarrow0$.  Contrapositively, if there exists any continuous length function for which $\ell(f^n)/n\not\rightarrow0$, then $f$ is undistorted.
		\item Suppose $H$ is a subgroup of $G$ equipped with a topology at least as fine as the inherited subspace topology.  If $d_G$ and $d_H$ are maximal pseudometrics on $G$ and $H$ respectively, then $d=d_H+d_G$ is a continuous right-invariant pseudometric on $H$.  We have $d\leq d_H$, and $d\leq K\cdot d_H+C$ for some constants $K$ and $C$ by the maximality of $d_H$, and hence $d$ and $d_H$ are quasi-isometric.  It follows that if $f\in H$ and $f$ is undistorted in $G$, then $f$ is also undistorted in $H$.
\end{enumerate}

In \cite{mann_rosendal_2018a}, Mann and Rosendal characterized the maximal metrics on groups of the form $\Homeo_0(M)$, which denotes the connected component of identity in the group of all homeomorphisms of a compact connected manifold $M$\textemdash the maximal metrics are quasi-isometric to the fragmentation metrics given by finite open coverings of $M$.  Earlier in \cite{militon_2014a}, Militon studied distortion elements of $\Homeo_0(M)$, where $M$ is a compact surface, in terms of fragmentation length.  By the results of \cite{mann_rosendal_2018a}, this approach is equivalent to studying distortion in the sense of maximal pseudometrics.

For $k\geq 1$ an integer, let $\Diff_+^k(\mathbb{S}^1)$ (resp. $\Diff_+^k(I)$) denote the Polish group of orientation-preserving $C^k$ diffeomorphisms of the circle $\mathbb{S}^1$ (resp. the interval $I=[0,1]$), equipped with the uniform $C^k$ topology.  The following subgroup inclusions hold:\\

\begin{center} $\Diff_+^1(\mathbb{S}^1)\geq \Diff_+^2(\mathbb{S}^1)\geq \Diff_+^3(\mathbb{S}^1) \geq ...$
\end{center}
\vspace{.3cm}

The topology on each successive group is strictly finer than the subspace topology inherited from its predecessor in the list, and therefore the notions of distortion associated to each group are \textit{a priori} distinct.  Let us say that a diffeomorphism $f$ of class $C^k$ is \textit{$C^k$-distorted} if $f$ is either distorted in the Polish group $\Diff_+^k(\mathbb{S}^1)$ (or $\Diff_+^k(I)$), or if $f$ is torsion.  (We allow torsion elements to be regarded as $C^k$-distorted to satisfy our intuition.)  By remark (4) above, if $f$ is $C^j$-undistorted then $f$ is $C^k$-undistorted for all $k\geq j$; and contrapositively, if $f$ is $C^k$-distorted, then $f$ is $C^j$-distorted for all $j\leq k$.  Put in other words, if $f$ is a circle diffeomorphism with a high degree of smoothness (for instance $C^\infty$), then $f$ has a greater chance to be $C^k$-undistorted for higher values of $k$.

The purpose of this paper is to initiate a study of distortion of circle diffeomorphisms from the perspective of maximal pseudometrics.  In \cite{cohen_2016a}, the author showed that each of the diffeomorphism groups $\Diff_+^k(\mathbb{S}^1)$ (resp. $\Diff_+^k(I)$) admits a maximal metric; and in case $k=1$, a (pseudo-) metric is given explicitly.  The maximal metric provided in \cite{cohen_2016a} for $\Diff_+^1(\mathbb{S}^1)$ is actually erroneous, because it is not right-invariant.  In an appendix to the present paper, we correct this error.  Using this explicit pseudometric (see Figure \ref{fig1}), our first theorem characterizes $C^1$-distortion of circle diffeomorphisms.  The proof is in Section 2.

\begin{thm} \label{thm1}  An element $f$ in $\Diff_+^1(\mathbb{S}^1)$ or $\Diff_+^1(I)$ is $C^1$-undistorted if and only if $f$ has a hyperbolic periodic point; i.e., a point $x\in\mathbb{S}^1$ with $f^q(x)=x$ and $(f^q)'(x)\neq 1$, for $q\in\mathbb{Z}^+$.
\end{thm}

The problem of classifying the $C^k$-distorted elements for $k\geq 2$ seems complicated.  Since $C^k$-distortion is a conjugacy invariant, it is natural to separately consider the three different conjugacy types of circle diffeomorphisms: (I) diffeomorphisms with rational rotation number and at least one hyperbolic periodic point; (II) diffeomorphisms with rational rotation number and no hyperbolic periodic point; and (III) diffeomorphisms with irrational rotation number.  Case I is settled by Theorem \ref{thm1} above; such maps are $C^k$-undistorted for every $k\geq1$.

\begin{cor}  For any $k\geq 1$, if $f\in\Diff_+^k(\mathbb{S}^1)$ (resp. $f\in\Diff_+^k(I)$) has a hyperbolic periodic point (resp. hyperbolic fixed point), then $f$ is $C^k$-undistorted.
\end{cor}

In this article, we focus especially on building some understanding of case II.  If $f$ has rotation number $p/q\in\mathbb{Q}$ and no hyperbolic periodic point, then $f^q$ has rotation number $0$ and no hyperbolic fixed point.  Clearly $f$ is distorted if and only if $f^q$ is distorted.  So to classify undistorted elements among those with rational rotation number, it suffices only to consider rotation number $0$.  Although we do not have a full classification in any case, we would like to provide some criteria and examples.  

Firstly, given the statement of Theorem 1, it is tempting to conjecture that a smooth diffeomorphism $f$ might necessarily be $C^2$-undistorted if there exists a ``$2$nd-order hyperbolic fixed point'' for $f$, i.e. a fixed point $x\in\mathbb{S}^1$ with $f''(x)\neq 0$.  The following example rules out this possibility.

\begin{prop} \label{thm2}  There exists an analytic circle diffeomorphism $f$ with rotation number $0$, which is $C^k$-distorted for every $k\geq 1$, and whose only fixed point $x\in\mathbb{S}^1$ satisfies $f''(x)\neq 0$.
\end{prop}

\begin{center}
\begin{figure}[b] \label{fig1}
\renewcommand{\arraystretch}{2}
\begin{tabular}{|c|c|}
\hline
Group & Maximal Pseudometric\\ [2ex]
\hline
$\Diff_+^1(\mathbb{S}^1)$ & $d(f,g)=\displaystyle\sup_{x\in\mathbb{S}^1}|\log f'(x)-\log g'(x)|$\\ [2ex]
\hline
$\Diff_+^{1+AC}(\mathbb{S}^1)$ & $d(f,g)=\displaystyle\int_{\mathbb{S}^1}\left|\dfrac{f''}{f'}-\dfrac{g''}{g'}\right|$\\ [2ex]
\hline
$\Diff_+^1(I)$ & $d(f,g)=\displaystyle\sup_{x\in\mathbb{S}^1}|(\log f'(x)-\log f'(0))-(\log g'(x)-\log g'(0))|$\\ [2ex]
\hline
$\Diff_+^{1+AC}(I)$ & $d(f,g)=\displaystyle\int_{I}\left|\dfrac{f''}{f'}-\dfrac{g''}{g'}\right|$\\ [2ex]
\hline
\end{tabular}
\vspace{.3cm}
\caption{Explicit maximal pseudometrics on Polish groups of diffeomorphisms.  The third and fourth are genuine metrics.  The first and second can be made into maximal metrics by adding a uniform distance term $\displaystyle\sup_{x\in\mathbb{S}^1}|f(x)-g(x)|$.}
\end{figure}
\end{center}

We also wish to exhibit interesting examples of $C^k$-undistorted diffeomorphisms, but we are somewhat hindered by the fact that we currently lack an explicit closed form for a maximal pseudometric on $\Diff_+^k(\mathbb{S}^1)$, $k\geq 2$.  To bridge this difficulty, we employ a Polish group of ``intermediate smoothness,'' introduced in \cite{cohen_2018a}: we denote by $\Diff_+^{1+AC}(\mathbb{S}^1)$ (resp. $\Diff_+^{1+AC}(I)$) the subgroup of $\Diff_+^1(\mathbb{S}^1)$ (resp. $\Diff_+^1(I)$) consisting of diffeomorphisms whose first derivative is absolutely continuous.  We have $\Diff_+^1(\mathbb{S}^1)\geq\Diff_+^{1+AC}(\mathbb{S}^1)\geq\Diff_+^2(\mathbb{S}^1)$, and in \cite{cohen_2018a} the author has shown that the group topology on $\Diff_+^{1+AC}(\mathbb{S}^1)$ refines the $C^1$-topology, but is coarser than the $C^2$-topology.  Thus by our previous remarks, if $f\in\Diff_+^{1+AC}(\mathbb{S}^1)$ is $C^{1+AC}$-undistorted (i.e., undistorted in the Polish group $\Diff_+^{1+AC}(\mathbb{S}^1)$), then it is also $C^k$-undistorted for all $k\geq 2$.  In Section 3, we give a simple criterion for identifying non-distortion at the $C^{1+AC}$ level, and we give the following example.

\begin{thm} \label{thm3}  There exists an analytic circle diffeomorphism $f$ with rotation number $0$, with no hyperbolic fixed point, such that $f$ is $C^{1+AC}$-undistorted (and hence $C^k$-undistorted for all $k\geq 2$).  Moreover, this diffeomorphism may be taken arbitrarily close to identity in the $C^{1+AC}$ topology.
\end{thm}

This affirmatively answers Question 2 in the article \cite{navas_2018a} of Navas (see Remark \ref{rem_navas}).

Theorem \ref{thm1} implies that the $C^1$-distorted elements of $\Diff_+^1(\mathbb{S}^1)$ comprise a simple closed set; on the other hand Theorem \ref{thm3} implies that the set of distortion elements of $\Diff_+^k(\mathbb{S}^1)$ for $k>1$ is more complicated.

\begin{cor} \label{cor1}  Let $k$ be an integer $\geq2$, or let $k=1+AC$.  Then the set of $C^k$-distorted elements of $\Diff_+^k(\mathbb{S}^1)$ is not closed.  Also, the set of $C^{1+AC}$-distorted elements is not open in $\Diff_+^{1+AC}(\mathbb{S}^1)$.
\end{cor}

We would like to know if it is always possible to find more and more undistorted diffeomorphisms at higher degrees of smoothness.

\begin{question}  For each integer $k\geq 2$, does there exist a diffeomorphism which is $C^k$-distorted but $C^{k+1}$-undistorted?
\end{question}

In Section 4, we prove the maximality of (pseudo-) metrics we use for $\Diff_+^{1+AC}(\mathbb{S}^1)$ and $\Diff_+^{1+AC}(I)$, which are listed for the reader in Figure \ref{fig1}.  We denote each metric simply $d$, since the choice of metric is always clear in context.  We also explicitly compute the quasi-isometry type of these groups.

\begin{thm}  Each of the groups $\Diff_+^{1+AC}(\mathbb{S}^1)$ and $\Diff_+^{1+AC}(I)$ is quasi-isometric to a hyperplane of the Banach space $L^1(I)$.
\end{thm}

Lastly, we would like to make some remarks on case III, the problem of classifying $C^k$-distorted circle diffeomorphisms among those with irrational rotation number, which appears to be related to the classical $C^k$-linearization problems that have inspired a substantial literature.  If $f$ is a $C^k$ circle diffeomorphism with irrational rotation number, then it has been shown by various authors that $f$ is $C^k$-conjugate to a rotation if and only if its set of iterates $\{f^j:j\geq 0\}$ have uniformly bounded derivatives of all orders up through $k$ (see for instance \cite{katznelson_ornstein_1989a} Theorem 2.1).  Combining this with the characterization of coarsely bounded sets provided in \cite{cohen_2016a}, we see that $f$ is $C^k$-conjugate to a rotation if and only if its set of iterates $\{f^j:j\geq0\}$ is a coarsely bounded set in $\Diff_+^k(\mathbb{S}^1)$.  If $f$ is $C^k$-undistorted, then the magnitudes of the derivatives of its iterates are not only unbounded, but in some sense grow linearly-- thus \textit{a priori}, to be $C^k$-undistorted is a strong way of being non-$C^k$-linearizable.

Arnol'd \cite{arnold_1961a} showed that if $f$ is an analytic circle diffeomorphism whose rotation number $\alpha$ satisfies a certain Diophantine condition (namely, that $|\alpha-\frac{p}{q}|>\frac{K}{q^{2+\beta}}$ for some constants $K,\beta$, for all $\frac{p}{q}\in\mathbb{Q}$), then $f$ is analytically conjugate to the rigid rotation of the circle through angle $\alpha$.  Since rotations are distorted and distortion is a conjugacy invariant, we see that all such diffeomorphisms $f$ are distorted in $\Diff_+^k(\mathbb{S}^1)$, for all $k\geq 1$.  On the other hand, Arnol'd gave examples of analytic diffeomorphisms with irrational rotation number which are not $C^1$-linearizable.  At the moment we are unable to provide an example of an undistorted aperiodic circle diffeomorphism.

\begin{question}  Does there exist a circle diffeomorphism with irrational rotation number which is $C^k$-undistorted, for any $k\geq 2$?
\end{question}

We remark that Lemma 1 of Navas in \cite{navas_2018a} shows that every aperiodic circle diffeomorphism of class $C^{1+AC}$ is $C^{1+AC}$-distorted.

\section{Classification of $C^1$-Distortion}

In this section we fix the pseudometric $d$ on $\Diff_+^1(\mathbb{S}^1)$ listed in Figure \ref{fig1}; its maximality is proven in the appendix.  So to determine if a diffeomorphism $f$ is $C^1$-distorted, we want to compute whether the distance $d(f^n,e)$ grows sublinearly, i.e. $f$ is $C^1$-distorted if

\begin{center} $\displaystyle\lim_{n\rightarrow\infty}\frac{1}{n}\cdot\sup_{x\in\mathbb{S}^1}|\log (f^n)'(x)|=\lim_{n\rightarrow\infty}\frac{1}{n}\cdot\sup_{x\in\mathbb{S}^1}\left|\sum_{k=0}^{n-1}\log f'(f^k(x))\right|=0$.
\end{center}
\vspace{.3cm}

Let us make some remarks to clarify the relationship of this notion of distortion to some that have appeared previously in the literature (see also Remark \ref{rem_navas} in the next section).

\begin{rem} In the classic text \cite{demelo_vanstrien_1993a} \S I.2, Demelo and van Strien define the \textit{distortion} of a circle diffeomorphism $f$ to be the quantity

\begin{center} $\mbox{\rm Dist}(f,\mathbb{S}^1)=\displaystyle\sup_{x,y\in\mathbb{S}^1}(\log f'(x)-\log f'(y))$.
\end{center}
\vspace{.3cm}

They then look for uniform bounds on the distortion of iterates $\mbox{\rm Dist}(f^n,\mathbb{S}^1)$ in order to prove Denjoy's theorem and many other results.  Note that by the mean value theorem, there is $y\in\mathbb{S}^1$ with $f'(y)=1$, and therefore $\mbox{\rm Dist}(f^n,\mathbb{S}^1)\geq d(f^n,e)$.  Thus if $\mbox{\rm Dist}(f,\mathbb{S}^1)$ is uniformly bounded by a constant $K$ for all $n$, so too is $d(f^n,e)$.  This means if a diffeomorphism $f$ has uniformly bounded distortion in the sense of Demelo-van Strien, then the iterates $\{f^n:n\in\mathbb{N}\}$ comprise a coarsely bounded set in $\Diff_+^1(\mathbb{S}^1)$, and hence $f$ is $C^1$-distorted in our sense.  The converse is not true, except when $f$ is $C^1$-conjugate to a rotation.
\end{rem}

\begin{rem} \label{rem_discrete} In \cite{calegari_freedman_2006a}, Calegari and Freedman called a circle diffeomorphism $f$ \textit{distorted} if there exists a finitely generated subgroup $\Gamma\leq\Diff_+^1(\mathbb{S}^1)$ in which the subgroup inclusion $\langle f\rangle\hookrightarrow\Gamma$ is not a quasi-isometric embedding.  A similar definition is used for surface diffeomorphisms by Franks and Handel \cite{franks_handel_2006a} and by Militon \cite{militon_2014a}.  If such a group $\Gamma$ exists, in this paper we will say that $f$ is \textit{discrete-distorted}.  Calegari and Freedman showed that each rigid rotation of the circle is discrete-distorted.

Equipping $\Gamma$ with the discrete topology, and applying remark (4) from our introduction, we see that a circle diffeomorphism which is discrete-distorted is also $C^k$-distorted for every $k\geq 1$.  Contrapositively, it is interesting to construct $C^k$-undistorted diffeomorphisms (which we do in Section 3), because they cannot embed non-quasi-isometrically into any finitely generated subgroup of $\Diff_+^k(\mathbb{S}^1)$.
\end{rem}

We now turn to our classification of $C^1$-distortion in the sense of maximal metrics.

\begin{proof}[Proof of Theorem \ref{thm1}]  We argue only for the case of the circle; the arguments for interval diffeomorphisms are essentially identical.\\

\textit{Case I: $f$ has a hyperbolic periodic point.}  If $f^q(x_0)=x_0$ and $(f^q)'(x_0)=K\neq 0$ for some $q\geq1$ and $x_0\in\mathbb{S}^1$, then by applying the chain rule, we have

\begin{align*}
\displaystyle\lim_{n\rightarrow\infty}\frac{d(f^n,e)}{n} &= \lim_{n\rightarrow\infty}\frac{d(f^{qn},e)}{qn}\\
&\geq \displaystyle\lim_{n\rightarrow\infty}\frac{|\log ((f^q)^n)'(x_0)|}{qn}\\
&= \displaystyle\lim_{n\rightarrow\infty}\frac{1}{qn}\left|\sum_{i=0}^{n-1} \log (f^q)'(f^{qn}(x_0))\right|\\
&= \frac{1}{q}\cdot|K|>0,
\end{align*}

\noindent so $f$ is $C^1$-undistorted.\\

\textit{Case II: $\rot(f)=p/q\in\mathbb{Q}$ and $f$ has no hyperbolic periodic point.}  As we mentioned in the introduction, $f$ is distorted if and only if $f^q$ is distorted; for this reason, it suffices for us to assume that $\rot(f)=0$ and $f$ has no hyperbolic fixed point, by replacing $f$ with $f^q$ if necessary.

First let us establish that if $[a,b]$ is any subarc of $\mathbb{S}^1$ with $f(a)=a$, $f(b)=b$, $f'(a)=f'(b)=1$ and $f(x)\neq x$ for all $x\in(a,b)$, then

\begin{center} $\displaystyle\lim_{n\rightarrow\infty}\frac{1}{n}\cdot\sup_{x\in[a,b]}|\log(f^n)'(x)|=0$.
\end{center}
\vspace{.3cm}

To see this, let $\epsilon>0$.  Choose $y\in(a,b)$ arbitrarily and set $J_0=[y,f(y)]$ or $J_0=[f(y),y]$ (depending on the order of $y$ and $f(y)$.  Set $J_k=f^k(J_0)$ for $n\in\mathbb{Z}$, so $[a,b]=\bigcup_{j=-\infty}^\infty J_k$.  Let $M_k=\sup\{|\log f'(x)|:x\in J_k\}$.  Note that $\displaystyle\lim_{k\rightarrow\infty}M_k=\lim_{k\rightarrow-\infty}M_k=0$ by the continuity of $f'$.  So we may find $N\in\mathbb{N}$ so large that $M_n<\epsilon$ for all $n\in\mathbb{Z}$ with $|n|>N$.  For any $x\in [a,b]$, there are at most $2N+1$ points $f^k(x)$ in the orbit of $x$ which satisfy $f^k(x)\in\bigcup_{j=-N}^N J_j$, and therefore $|\log (f^n)'(x)|=\left|\displaystyle\sum_{k=0}^{n-1}\log f'(f^k(x))\right|\leq \sum_{j=-N}^N M_j+(n-2N-1)\epsilon=K+n\epsilon$.  Thus $\displaystyle\lim_{n\rightarrow\infty}\frac{1}{n}\cdot\sup_{x\in[a,b]}|\log(f^n)'(x)|\leq\lim_{n\rightarrow\infty}\frac{K}{n}+\epsilon=\epsilon$.  Since $\epsilon$ was arbitrary, this proves the claim above.

Now for the sake of a contradiction, suppose $\displaystyle\lim_{n\rightarrow\infty}\frac{d(f^n,e)}{n}>K>0$, for some constant $K$.  It means that for each sufficiently large $n$, there exists a point $x_n\in \mathbb{S}^1$ so that $|\log (f^n)'(x_n)|\geq Kn$.  By passing to a subsequence, without loss of generality we may assume $\log (f^{n_i})'(x_{n_i})\geq Kn_i$ for all $i\geq 1$.  (If there are only finitely many positive terms $\log (f_n)'(x_n)$, then there are infinitely many negative ones, in which case we can replace $f$ with $f^{-1}$ to find our subsequence).  Since $\displaystyle\sum_{j=0}^{n_i-1}\log f'(f^{j}(x_{n_i}))\geq Kn_i$, we deduce there exists some $j_i\in\{0,...,n_i-1\}$ for which $\log f'(f^{j_i}(x_{n_i}))\geq K$.

Let $\{I_k\}$ denote the countable set of all maximal open subintervals of $\mathbb{S}^1$ on which $f$ has no fixed point.  By our previous claim, no subsequence of the points $x_{n_i}$ may lay in a single subinterval $I_k$.  Therefore there are infinitely many subintervals $I_k$, and each contains only finitely many of the points $x_{n_i}$.  Consequently, by passing to a further subsequence which we again denote $x_{n_i}$, we assume that $x_{n_i}\in I_{k_i}$ where the intervals $\{I_{k_i}:i\in\mathbb{N}\}$ are pairwise disjoint.  Passing to a subsequence once more, using the compactness of $\mathbb{S}^1$, we assume that $x_{n_i}\rightarrow x_0\in \mathbb{S}^1$.  

For each $i$ let $a_i$ denote the left endpoint of $I_{k_i}$, so $f(a_i)=a_i$.  Note that the points $x_{n_i}$, $f^{j_i}(x_{n_i})$, and $a_i$ all lie in the closure of the same subinterval $I_{k_i}$.  Since the diameters of $I_{k_i}$ tend to $0$ with $i$, we get that $f^{j_i}(x_{n_i})\rightarrow x_0$ and $a_i\rightarrow x_0$ as well.  Since $f'$ is continuous and $f^{j_i}(x_{n_i})\rightarrow x_0$, we have $\log f'(x_0)\geq K$.  Since $f$ is continuous and $a_i\rightarrow x_0$, we get $f(x_0)=x_0$, so $x_0$ is a hyperbolic fixed point after all, a contradiction.\\

\textit{Case III: $\rot(f)=\alpha\notin\mathbb{Q}$.}  In this case $f$ is uniquely ergodic (\cite{walters_2000a} Theorem 6.18), and therefore the functions $\frac{1}{n}\log(f^n)'=\frac{1}{n}\cdot\displaystyle\sum_{i=0}^{n-1}f'\circ f^i$ converge uniformly to some constant $L$ (see \cite{walters_2000a} Theorem 6.19).

Suppose for a contradiction that $L\neq 0$.  Set $\epsilon=|L|/2$.  Then for some $n$, for every $x\in\mathbb{S}^1$, we have\\

\begin{center} $n(L-\epsilon)<\log(f^n)'(x)<n(L+\epsilon)$.
\end{center}
\vspace{.3cm}

Thus, depending on the sign of $L$, we have either $\log(f^n)'$ is $>0$ everywhere or $<0$ everywhere.  In other words either $(f^n)'(x)>1$ for all $x$, or $(f^n)'(x)<1$ for all $x$, a contradiction since $f^n$ is a diffeomorphism.  Therefore $\frac{1}{n}\log(f^n)'\rightarrow0$ uniformly, so $\displaystyle\dfrac{d(f^n,e)}{n}\rightarrow0$ and $f$ is $C^1$-distorted.
\end{proof}

\section{Distortion in Class $C^{1+AC}$}

We now turn to the study of distortion of circle diffeomorphisms in class $C^{1+AC}$.  We fix the metric $d$ on $\Diff_+^{1+AC}(\mathbb{S}^1)$ listed in Figure \ref{fig1}.  For $f,g\in\Diff_+^{1+AC}(\mathbb{S}^1)$, the second derivatives $f''$ and $g''$ are defined almost everywhere, and we note the following relation:

\begin{center} $\dfrac{(f\circ g)''}{(f\circ g)'}=\left(\dfrac{f''}{f'}\circ g\right)\cdot g'+\dfrac{g''}{g'}$.
\end{center}
\vspace{.3cm}

From the above, one deduces the following formula for the compositional iterates of $f$:

\begin{center} $\dfrac{(f^n)''}{(f^n)'}=\displaystyle\sum_{k=0}^{n-1}\left(\dfrac{f''}{f'}\circ f^k\right)\cdot(f^k)'$.
\end{center}
\vspace{.3cm}

So $f$ is $C^{1+AC}$-distorted if and only if $\displaystyle\lim_{n\rightarrow\infty}\frac{1}{n}\int_{\mathbb{S}^1}\left|\sum_{k=0}^{n-1}\left(\dfrac{f''}{f'}\circ f^k\right)\cdot(f^k)'\right|=0$.

We are now ready to give the example promised in Proposition \ref{thm2}, which is essentially the same as Example 1.12 in \cite{franks_handel_2006a}.

\begin{thm:thm2}  There exists an analytic circle diffeomorphism $f$ with rotation number $0$, which is $C^k$-distorted for every $k\geq 1$, and whose only fixed point $x\in\mathbb{S}^1$ satisfies $f''(x)\neq 0$.
\end{thm:thm2}

\begin{proof}  For this proof, we imagine $\mathbb{S}^1$ as the unit circle in the complex plane.  This circle is in bijection, via the map $\varphi(z)=i(1-z)/(1+z)$, with the one-point compactification $\mathbb{R}\cup\{\infty\}$ of the real line.  The group $\mathcal{M}$ of M\"obius transformations, i.e. rational maps of the form $m(r)=(ar+b)/(cr+d)$ with $a,b,c,d\in\mathbb{R}$, acts on $\mathbb{R}\cup\{\infty\}$.  We let $F,G\in\mathcal{M}$ be defined by $F(r)=r/(r+1)$, $G(r)=r/2$, and we let $f=\varphi\circ F\circ \varphi^{-1}$, $g=\varphi\circ G\circ\varphi^{-1}$, so $f$ and $g$ are analytic circle diffeomorphisms.  The map $F$ has a single fixed point at $r=0$, so $f$ has a single fixed point at $z=\varphi^{-1}(0)=1\in\mathbb{S}^1$.

Let $\ell$ denote the word length function in the finitely generated group of circle diffeomorphisms $\Gamma=\langle f,g\rangle$, which is isomorphic to $\langle F,G\rangle\leq\mathcal{M}$.  We note that $gfg^{-1}=f^2$, so $g^nfg^{-n}=f^{2^n}$ for each $n\geq 1$.  Therefore $\frac{\ell(f^{2^n})}{2^n}\leq \frac{2n+1}{2^n}\rightarrow 0$ as $n\rightarrow\infty$, so $f\hookrightarrow\Gamma$ is not a quasi-isometric embedding.  Thus $f$ is discrete-distorted, and hence $C^k$-distorted for every $k$.

To check that the second derivative of our map is nonzero at $z=1$, we choose a chart $\psi$ in a neighborhood $U$ of $z=1\in\mathbb{S}^1$ defined by $\psi(z)=-i\log z$, where here $\log$ denotes an appropriate branch of the complex logarithm.  Then we want to verify that the second derivative of $\psi\circ f\circ \psi^{-1}$ evaluated at $x=\psi(1)=0$ is nonzero.  For this, we compute directly that\\

\begin{center} $f(z)=\dfrac{z-1+2iz}{z-1+2i}$, $f'(z)=\dfrac{-4}{(z-1+2i)^2}$, and $\frac{\psi''}{\psi'}(z)=-\dfrac{1}{z}$,
\end{center}
\vspace{.3cm}

\noindent and therefore\\

\begin{align*}
\frac{(\psi\circ f\circ \psi^{-1})''}{(\psi\circ f\circ\psi^{-1})'}(0) &= \left[\left(\frac{\psi''}{\psi'}\circ f\right)\cdot f'+\frac{f''}{f'}-\frac{\psi''}{\psi'}\right]\cdot\left(\frac{1}{\psi'}\right)\circ\psi^{-1}(0)\\
&= \left[\left(-\frac{z-1+2i}{z-1+2iz}\cdot\frac{-4}{(z-1+2i)^2}\right)_{z=1}+\left(\frac{-2}{z-1+2i}\right)_{z=1}\right.\\
& \phantom{{}=1}\left.-\left(-\frac{1}{z}\right)_{z=1}\right]\cdot\left(-\frac{i}{z}\right)_{z=1}\\
&= 1\neq 0.
\end{align*}
\end{proof}


\begin{rem} \label{rem_navas}  In \cite{navas_2018a}, Navas defines the \textit{asymptotic distortion} of an interval diffeomorphism $f$ whose first derivative has bounded variation to be the quantity

\begin{center} $\displaystyle\lim_{n\rightarrow\infty}\frac{1}{n}V(\log (f^n)')$,
\end{center}
\vspace{.3cm}

\noindent where $V$ denotes the total variation.  In the case that the first derivative of $f$ is absolutely continuous, then the total variation of $\log (f^n)'$ is equal to $d(f^n,e)=\int_I|\frac{f''}{f'}|$ (see \cite{leoni_2009a} Exercise 2.3 (ii)).  So $f$ has nonzero asymptotic distortion in the sense of Navas if and only if $f$ is $C^{1+AC}$-undistorted.  Thus, our example of a $C^{1+AC}$-undistorted analytic circle diffeomorphism with no hyperbolic fixed point in Theorem \ref{thm3} gives a positive answer to Question 2 of \cite{navas_2018a}.
\end{rem}

\begin{lemma}[Criterion for $C^{1+AC}$ Non-Distortion] \label{lemma_1+AC_undistorted}  Let $f\in\Diff_+^{1+AC}(\mathbb{S}^1)$ and suppose that there exists a subarc $[a,b]$ of $\mathbb{S}^1$ with the property that the intervals $[f^i(a),f^i(b)]$ ($i\in\mathbb{Z}$) are pairwise disjoint, and $f''\geq 0$ on $[f^i(a),f^i(b)]$ for every $i\in\mathbb{Z}$ (or $f''\leq 0$ on $[f^i(a),f^i(b)]$ for every $i\in\mathbb{Z}$).  Then

\begin{center}  $\displaystyle\lim_{n\rightarrow\infty}\dfrac{d(f^n,e)}{n}\geq \left|\sum_{i\in\mathbb{Z}}\int_{f^i(a)}^{f^i(b)}\dfrac{f''}{f'}\right|$.
\end{center}

\end{lemma}

\begin{rem}  Any $f$ which satisfies the criterion of Lemma \ref{lemma_1+AC_undistorted} necessarily has periodic points.  For if not, then $f$ is aperiodic of class $C^{1+AC}$, so by Denjoy's theorem $f$ is conjugate to a rotation.  Therefore $\mathbb{S}^1$ is covered by the intervals $[f^i(a),f^i(b)]$, and we get $f''\geq 0$ (or $f''\leq 0$) on all of $\mathbb{S}^1$.  Since $f'>0$, this implies $\frac{f''}{f'}\geq 0$ (resp. $\frac{f''}{f'}\leq 0$) on all of $\mathbb{S}^1$, which in turn implies that $\log f'$ is everywhere increasing (resp. everywhere decreasing), an impossibility for a diffeomorphism.
\end{rem}

\begin{proof}[Proof of Lemma \ref{lemma_1+AC_undistorted}]  Assume $f''\geq 0$ on $[f^i(a),f^i(b)]$ for every $i\in\mathbb{Z}$ (the case $f''\leq0$ is similar).  Since $(f^k)'>0$, we have $\left(\frac{f''}{f'}\circ f^k\right)(f^k)'\geq 0$ on $[f^i(a),f^i(b)]$ for every $i,k\in\mathbb{Z}$ as well.  Therefore\\

\begin{align*}
\lim_{n\rightarrow\infty}\dfrac{d(f^n,e)}{n} &= \lim_{n\rightarrow\infty}\dfrac{1}{n}\int_0^1\left|\dfrac{(f^n)''}{(f^n)'}\right|\\
&\geq \lim_{n\rightarrow\infty}\dfrac{1}{n}\sum_{i\in\mathbb{Z}}\int_{f^i(a)}^{f^i(b)}\left|\dfrac{(f^n)''}{(f^n)'}\right|\\
&\geq \lim_{n\rightarrow\infty}\dfrac{1}{n}\sum_{i\in\mathbb{Z}}\int_{f^i(a)}^{f^i(b)}\left|\sum_{k=0}^{n-1}\left(\dfrac{f''}{f'}\circ f^k\right)(f^k)'\right|\\
&= \lim_{n\rightarrow\infty}\dfrac{1}{n}\sum_{k=0}^{n-1}\sum_{i\in\mathbb{Z}}\int_{f^i(a)}^{f^i(b)}\left(\dfrac{f''}{f'}\circ f^k\right)(f^k)'\\
&= \lim_{n\rightarrow\infty}\dfrac{1}{n}\sum_{k=0}^{n-1}\sum_{i\in\mathbb{Z}}\int_{f^{i+k}(a)}^{f^{i+k}(b)}\dfrac{f''}{f'}\\
&= \lim_{n\rightarrow\infty}\dfrac{1}{n}\cdot n\cdot\sum_{i\in\mathbb{Z}}\int_{f^i(a)}^{f^i(b)}\dfrac{f''}{f'}\\
&= \sum_{i\in\mathbb{Z}}\int_{f^i(a)}^{f^i(b)}\dfrac{f''}{f'}.
\end{align*}
\end{proof}

\begin{thm:thm3}  Let $0<K\leq 2$ be arbitrary.  There exists an analytic diffeomorphism $f$ of $\mathbb{S}^1$ with the following properties:
\begin{itemize}
		\item $f$ has a single fixed point $x_0$ with $f'(x_0)=1$ and $f''(x_0)=0$;
		\item $\int_{\mathbb{S}^1}|f''|\leq 2K$; and
		\item $f$ is $C^{1+AC}$-undistorted.
\end{itemize}
\end{thm:thm3}

\begin{proof}  For this construction, we imagine $\mathbb{S}^1$ as the interval $[0,1]$ with the endpoints identified.  Let 

\begin{center} $\psi(x)=\frac{1}{2}(1-\cos2\pi x)$,
\end{center}
\vspace{.3cm}

\noindent so $\psi:\mathbb{R}\rightarrow\mathbb{R}$ is everywhere nonnegative, $\int_0^1\psi=1$, and $\psi(x)=\psi(1-x)$ for all $x$.  Let $m$ be a fixed (large) positive integer to be determined later.  We define for $x\in\mathbb{R}$:

\begin{align*}
f''(x) &= K\psi(x)-c_m(\psi(x))^m\\
f'(x) &= 1+\int_0^x f''(t)dt\\
f(x) &= \int_0^x f'(t)dt
\end{align*}

\noindent where $c_m$ is a positive constant chosen so that $c_m\int_0^1(\psi(t))^m dt=K$.  It is clear that $f:\mathbb{R}\rightarrow\mathbb{R}$ is real-analytic; we claim that $f'$ is periodic with period $1$, $f'(x)>0$ for all $x$, $f(0)=0$, and $f(1)=1$, so that $f$ induces an analytic circle diffeomorphism which fixes $x_0=0$.

\begin{center}
\begin{figure}[h] \label{fig3} \includegraphics[scale=0.4]{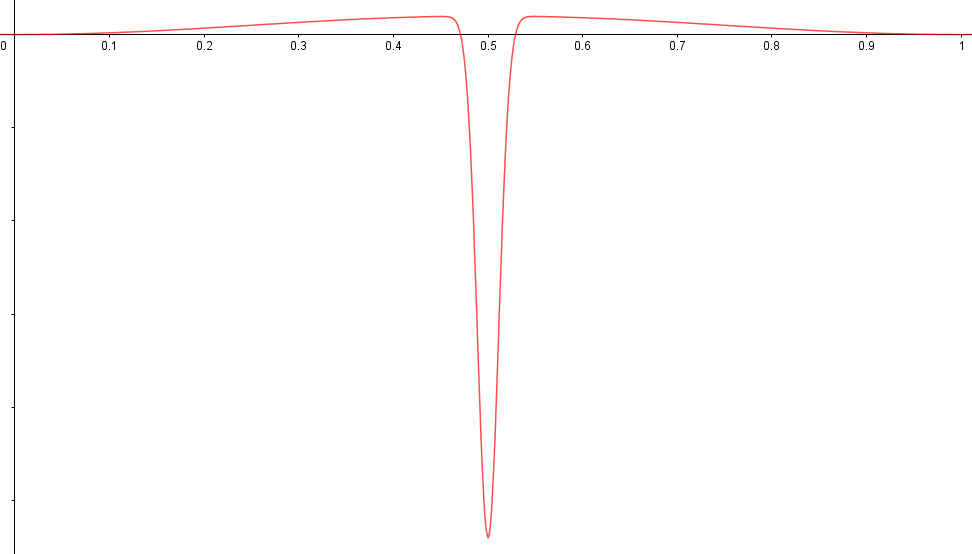}
\caption{The graph of $f''$, for $f$ a $C^{1+AC}$-undistorted circle diffeomorphism.}
\end{figure}
\end{center}

To see this, first observe that $f''$ is periodic with period $1$, satisfies $f''(x)=f''(1-x)$, $\int_0^1 f''(t)dt=0$, and has exactly four zeroes on $[0,1]$ located at $0$, $1$, and the following two points:

\begin{align*}
a_m &= \frac{1}{2\pi}\arccos\left(1-2\left(\frac{K}{c_m}\right)^{1/(m-1)}\right) &\in\left(0,\frac{1}{2}\right)\\
b_m &= 1-a_m &\in\left(\frac{1}{2},1\right)
\end{align*}
\vspace{.3cm}

It follows from these observations that $f'$ is periodic with period $1$, satisfies $f'(x)-1=-f'(1-x)+1$ with $f'(0)=f'(1)=0$, and is strictly increasing on $[0,a_m]$ and $[b_m,1]$, and strictly decreasing on $[a_m,b_m]$.  So $f'$ achieves its maximum at $a_m$ where $f'(a_m)<1+\int_0^{a_m} K\psi(x)dx< 1+\int_0^{1/2} K\psi(x)dx=1+K/2\leq 2$, and its minimum at $b_m$ where $f'(b_m)=2-f'(a_m)>0$.  This shows that $f'(x)>0$ for all $x$.  Lastly, it is clear that $f(0)=0$, and we have $f(1)=\int_0^{1/2}f'(x)dx+\int_{1/2}^1 f'(x)dx=\int_0^{1/2}f'(x)dx+\int_{1/2}^1(2-f'(1-x))dx=\int_0^{1/2}f'(x)dx-\int_{1/2}^0(2-f'(x))dx=1$.  So $f$ induces a circle diffeomorphism as claimed, which for simplicity we denote again by $f$.

Next we need an estimate on the size of $c_m$.  Observe, using the Taylor series for sine at $t=1/2$, that\\

\begin{align*}
\int_0^1(\psi(t))^mdt &= 2\int_{1/2}^{1}\left[\frac{1}{2}(1-\cos 2\pi t)\right]^mdt\\
&> 2\int_{1/2}^{1/2+1/\sqrt{m}}(\sin \pi t)^{2m}dt\\
&\geq 2\int_{1/2}^{1/2+1/\sqrt{m}}\left(1-\pi^2\left(t-\frac{1}{2}\right)^2\right)^{2m} dt\\
&\geq 2\int_{1/2}^{1/2+1/\sqrt{m}}\left(1-2m\pi^2\left(t-\frac{1}{2}\right)^2\right) dt\\
&= 2\left(\frac{1}{\sqrt{m}}+\frac{2\pi^2}{3\sqrt{m}}\right)\\
&> \frac{K}{\sqrt{m}},
\end{align*}
\vspace{.3cm}

\noindent and hence $c_m<\sqrt{m}$.  It follows that $a_m\geq \frac{1}{2\pi}\arccos\left(1-2\left(\frac{K}{\sqrt{m}}\right)^{1/(m-1)}\right)$, and since the terms on the right side of the inequality tend to $\frac{1}{2}$ as $m\rightarrow\infty$, we also have $a_m\rightarrow\frac{1}{2}$ from the left, and $b_m\rightarrow\frac{1}{2}$ from the right.

Set $p=\frac{20-K}{40}$.  Since we have $f''(x)\geq K\psi(x)-c_m p^m$ on $[0,p]$, we see that $f''\rightarrow K\psi$ uniformly from below on $[0,p]$ as $m\rightarrow\infty$.  Consequently, as $m\rightarrow\infty$, we have\\

\begin{center} $f\rightarrow \int_0^x(1+\int_0^t K\psi(s)ds)dt=x+\frac{K}{4}x^2+\frac{K}{8\pi^2}\cos2\pi x$ uniformly.
\end{center}
\vspace{.3cm}

The limiting map above sends $\frac{20-K}{40}$ to $\frac{1200-80K+K^2}{1600}+\frac{K}{8\pi^2}\cos 2\pi p>\frac{1200-80K}{1600}-\frac{K}{8\pi^2}>\frac{1200-80K}{1600}-\frac{K}{80}=\frac{1200-100K}{1600}\geq\frac{5}{8}>\frac{20+K}{40}>1-p$.  Therefore we may choose $m$ to be so large that\\

\begin{itemize}
		\item $f''$ is so uniformly close to $K\psi$ on $[0,p]$ that $f(p)>1-p$, and
		\item $a_m>p$ and $b_m<1-p$.
\end{itemize}
\vspace{.3cm}

In this way we guarantee that $f(a_m)>f(p)>p-1>b_m$, and so $f''$ is nonnegative on each interval of the form $[f^i(b_m),f^i(f(a_m))]$ ($i\in\mathbb{Z}$).  Thus Lemma \ref{lemma_1+AC_undistorted} applies and $f$ is undistorted, with\\

\begin{center}  $\displaystyle\lim_{n\rightarrow\infty}\dfrac{d(f^n,e)}{n}\geq \sum_{i\in\mathbb{Z}}\int_{f^i(b_m)}^{f^{i+1}(a_m)}\dfrac{f''}{f'}>0$.
\end{center}

Moreover we have $\int_0^{a_m}f''+\int_{b_m}^1 f''\leq\int_0^1 K\psi=K$.  Since $\int_0^1 f''=0$, we also have $-K\leq\int_{a_m}^{b_m}f''$, and so $\int_{\mathbb{S}^1}|f''|=\int_0^{a_m}f''-\int_{a_m}^{b_m}f''+\int_{b_m}^1f''\leq 2K$.
\end{proof}

\begin{cor:cor1}  Let $k$ be an integer $\geq2$, or let $k=1+AC$.  Then the set of $C^k$-distorted elements of $\Diff_+^k(\mathbb{S}^1)$ is not closed.  Also, the set of $C^{1+AC}$-distorted elements is not open in $\Diff_+^{1+AC}(\mathbb{S}^1)$.
\end{cor:cor1}

\begin{proof}  Take $f$ a $C^{1+AC}$-undistorted analytic diffeomorphism of $\mathbb{S}^1$ with $\rot(f)=0$, $f(x)>x$ for all $x\in(0,1)$, and without hyperbolic fixed points, as in Theorem \ref{thm2}.  Then $f$ is $C^k$-undistorted.  Given $\theta\in[0,1]$, let $R_\theta$ denote the rigid rotation of $\mathbb{S}^1$ through angle $2\pi\theta$. 

For $\epsilon<\displaystyle\sup_{x\in\mathbb{S}^1} (f(x)-x)$, the diffeomorphism $R_\epsilon\circ f$ is fixed-point free, so $\rot(R_\epsilon\circ f)\neq 0$.  Therefore the mapping $\theta\mapsto \rot(R_\theta\circ f)$, $[0,\epsilon]\rightarrow[0,\rot(R_\epsilon\circ f)]$ is a continuous monotone nondecreasing map onto a nontrivial closed subinterval of $[0,1)$ (see \cite{demelo_vanstrien_1993a} Lemma 4.1).  This interval $[0,\rot(R_\epsilon\circ f)]$ contains irrational numbers which satisfy the Diophantine condition of Arnol'd which we mentioned in the introduction.  If $\alpha$ is such a number, and $\delta\in[0,\epsilon]$ is such that $\rot(R_\delta\circ f)=\alpha$, then $R_\delta\circ f$ is analytically conjugate to a rotation.  Hence it is $C^k$-distorted, and since $\epsilon$ was arbitrary we see that distorted elements converge to $f$ in the $C^k$ topology.  This proves the first claim of the corollary.

By Theorem \ref{thm2}, since there are $C^{1+AC}$-undistorted elements arbitrarily close to identity in the $C^{1+AC}$ topology, and the identity is trivially $C^{1+AC}$-distorted, the $C^{1+AC}$-undistorted elements do not form a closed set, which proves the second claim.
\end{proof}

\section{Maximal Pseudometrics and Quasi-Isometry Types}

In this section we establish the maximality of our metrics on the groups $\Diff_+^{1+AC}(M^1)$, where $M^1=I$ or $M^1=\mathbb{S}^1$, and we explicitly compute their quasi-isometry types.  For the necessary background on coarsely bounded sets, maximal pseudometrics, and other concepts relating to the coarse geometry of topological groups, we refer the reader to \cite{rosendal_2018a}.  For the entire section, we think of $\mathbb{S}^1$ as the interval $[0,1]$ with the endpoints identified.

In the article \cite{cohen_2018a}, it was shown that the group $\Diff_+^{1+AC}(M^1)$ admits a unique Polish topology which is metrized by the following:

\begin{center} $\rho(f,g)=\displaystyle\sup_{x\in M^1}|f(x)-g(x)|+\sup_{x\in M^1}|f'(x)-g'(x)|+\int_{M^1}|f''-g''|$.
\end{center}
\vspace{.3cm}

We also topologize the Banach space $L^1(I)$ with its usual norm which we denote by $\|\cdot\|_1$.

\begin{lemma} \label{1+AC_mapping}  The mapping $\Phi:\Diff_+^{1+AC}(M^1)\rightarrow L^1(M^1)$, $\Phi(f)=\frac{f''}{f'}$ is continuous.  If $M^1=I$ then $\Phi$ is a bijection.  If $M^1=\mathbb{S}^1$, then $\Phi$ is a surjection onto the hyperplane of $L^1(\mathbb{S}^1)$ defined by $Y=\{H\in L^1(\mathbb{S}^1):\int_{\mathbb{S}^1}F=0\}$, and the preimage of each point in $Y$ is a left coset of the group of rotations.
\end{lemma}

\begin{proof}  If $f_n\rightarrow f$ in the group, it means $(f_n')$ is a sequence of strictly positive functions such that $f_n'\rightarrow f'$ uniformly, and $f_n''\rightarrow f''$ in the $L^1$-metric.  From the first condition we have that $\frac{1}{f_n'}\rightarrow\frac{1}{f'}$ uniformly.  Consequently, $\Phi$ is continuous.  

In case $M^1=I$, let $f,g\in\Diff_+^{1+AC}(I)$ with $f\neq g$.  Then $fg^{-1}\neq e$, so $\log(fg^{-1})'=[\log f'-\log g']\circ g^{-1}$ is not identically zero.  On the other hand the mean value theorem guarantees at least one point $x\in I$ where $\log(fg^{-1})'(x)=0$, so $\log f'-\log g'$ is nonconstant.  Therefore its derivative $\Phi(f)-\Phi(g)$ is nonzero.  This shows $\Phi$ is injective.

We compute the inverse of $\Phi$.  Let $H\in L^1(I)$, and define $F:I\rightarrow\mathbb{R}$, $f:I\rightarrow I$ by the rules\\

\begin{center} $F(x)=\int_0^x H(t)dt$; $C=\int_I exp(F)$; $f(x)=\frac{1}{C}\int_0^x exp(F(t))dt$.
\end{center}
\vspace{.3cm}

We have $f(0)=0$, $f(1)=1$, and $f'>0$, so $f\in\Diff_+^1(I)$.  By construction $F$ is absolutely continuous, and therefore so is $\frac{1}{C}\exp(F)=f'$, so in fact $f\in\Diff_+^{1+AC}(I)$.  It is easy to check that $\Phi(f)=H$.  

In case $M^1=\mathbb{S}^1$, given $H\in Y$, we can construct $f\in\Diff_+^{1+AC}(I)$ as above.  We hope to verify that in fact $f'(0)=f'(1)$, which would imply $f\in\Diff_+^{1+AC}(\mathbb{S}^1)$, ensuring that $\Phi$ is surjective.  For this, observe that $F(1)=0$ since $H\in Y$.  Therefore $F(0)=F(1)$, so $f'(0)=\frac{1}{C}=f'(1)$.

Suppose $f,g\in\Diff_+^{1+AC}(\mathbb{S}^1)$ and $\Phi(f)=\Phi(g)$.  Since $\Phi(f)-\Phi(g)=0$ we have that $[\log f'-\log g']\circ g^{-1}$ is a constant function.  So $(fg^{-1})'$ is constant, whence $fg^{-1}$ is a rigid rotation of the circle.
\end{proof}

The lemma above implies that the pseudometric $d$ on $\Diff_+^{1+AC}(M^1)$ defined by

\begin{center} $d(f,g)=\|\Phi(f)-\Phi(g)\|_1=\displaystyle\int_{M^1}\left|\frac{f''}{f'}-\frac{g''}{g'}\right|$
\end{center}
\vspace{.3cm}

\noindent is continuous.  Also, it is right-invariant (because of the relation at the beginning of Section 3).  We note also that if $M^1=I$, then $d$ is a genuine metric and not merely a pseudometric.

\begin{lemma} \label{lemma_1+AC_coarsely_bounded}  A subset $A\subseteq\Diff_+^{1+AC}(M^1)$ is coarsely bounded if and only if $\displaystyle\sup_{f\in A}\int_{M^1}\left|\frac{f''}{f'}\right|<\infty$.
\end{lemma}

\begin{proof}  Let $A\subseteq\Diff_+^{1+AC}(M^1)$ be a coarsely bounded set, and let $U$ be the $d$-ball about identity of radius $1$.  It means that there is a finite set $F$ and an integer $r$ with $A\subseteq FU^r$.  Since $d$ is right-invariant, $FU^r$ is $d$-bounded, and hence so is $A$.  Therefore $\displaystyle\sup_{f\in A}\int_{M^1}|\frac{f''}{f'}|<\infty$.

On the other hand, suppose $A\subseteq\Diff_+^{1+AC}(M^1)$ satisfies the condition $\displaystyle\sup_{f\in A}\int_{\mathbb{S}^1}\left|\frac{f''}{f'}\right|=\sup_{f\in A}V(\log f')=M<\infty$.  Let us first consider the case where $M^1=I$.  Let $U$ be an arbitrary basic open set, so $U$ is a $d$-ball about identity of radius $\epsilon>0$.  For any given $f\in A$, we have that the total variation of $\log f'$ is $\leq M$, whence $\sup_{x\in I}\log f'(x)\leq M$.  It follows that $e^{-M}\leq f'\leq e^M$.

Let $N$ be a large integer to be determined later.  For each $0\leq i\leq N$, set $f_i=\frac{i}{N}\id+(1-\frac{i}{N})f$, where $\id$ denotes the identity diffeomorphism.  Verify easily that $f_i\in\Diff_+^1(I)$, $f_0=f$, and $f_N=\id$.  Compute that\\

\begin{center} $\dfrac{f_i''}{f_i'}=\dfrac{(N-i)f''}{(N-i)f'+i}$
\end{center}
\vspace{.3cm}

\noindent and therefore for almost every $x\in I$, we have\\

\begin{align*}
\left|\dfrac{f_i''(x)}{f_i'(x)}-\dfrac{f_{i-1}''(x)}{f_{i-1}'(x)}\right| &= \left|\dfrac{-Nf''(x)}{[(N-i)f'(x)+1][(N-i+1)f'(x)+i-1]}\right|\\
&= \left|\dfrac{f''(x)}{f'(x)}\right|\cdot\dfrac{N}{[Nf'(x)+i(1-f')][Nf'(x)+(i-1)(1-f'(x))]}\\
&\leq M\cdot \frac{1}{Ne^{-2M}}.
\end{align*}
\vspace{.3cm}

We choose $N$ so large that the last expression above is $<\epsilon$.  Letting $u_i=f_if_{i-1}$, we see that $d(u_i,e)=d(f_i,f_{i-1})=\int_I|\frac{f_i''}{f_i'}-\frac{f_{i-1}''}{f_{i-1}'}|<\epsilon$, so $u_i\in U$.  Also we have $f=u_1u_2...u_N$, so $f\in U^N$.  Since $f$ was arbitrary, $A\subseteq U^n$, and since $U$ was an arbitrary basic open set, we have shown that $A$ is coarsely bounded.

For the case of $M^1=\mathbb{S}^1$, let $H$ denote the stabilizer of $0$ in $\Diff_+^{1+AC}(\mathbb{S}^1)$, and let $K$ denote the compact subgroup of rotations.  Also define:\\

\begin{center} $A^*=\{a^*\in H:\exists t\in K ~ \exists a\in A ~ ta^*=a\}$.
\end{center}
\vspace{.3cm}

Since $\displaystyle\sup_{f\in A}\int_{\mathbb{S}^1}\left|\frac{f''}{f'}\right|<\infty$, we also have $\displaystyle\sup_{f\in A*}\int_{\mathbb{S}^1}\left|\frac{f''}{f'}\right|<\infty$.  Let $U\subseteq\Diff_+^{1+AC}(\mathbb{S}^1)$ be an open neighborhood of $e$, and let $V=U\cap H$.  $V$ contains a basic open $d$-ball in $H$.  Therefore, repeating the argument as in the interval case, we can find a positive integer $N$ so that $A^{*}\subseteq V^N\subseteq U^N$.  This shows that $A^{*}$ is coarsely bounded.  But $A\subseteq KA^{*}$.  It follows that $A$ is contained in a product of coarsely bounded sets, hence $A$ is coarsely bounded as claimed.
\end{proof}

\begin{thm}  The pseudometric on $\Diff_+^{1+AC}(M^1)$ defined by $d(f,g)=\displaystyle\int\left|\frac{f''}{f'}-\frac{g''}{g'}\right|$ is maximal.
\end{thm}

\begin{proof}  By \cite{rosendal_2018a} Proposition 2.52, it is enough to show that $d$ is right-invariant, coarsely proper, and large-scale geodesic.

The right-invariance of $d$ follows from the relation at the beginning of Section 4.  By Lemma \ref{lemma_1+AC_coarsely_bounded}, $d$ is coarsely proper.  Let $H$ be the stabilizer of $0$ in $\Diff_+^{1+AC}(M^1)$ and consider the restricted metric $d|_H$.  By Lemma \ref{1+AC_mapping}, $\Phi$ is an isometry of $(H,d|_H)$ onto a closed subspace of $L^1(M^1)$ equipped with its norm metric.  Thus $d|_H$ is a geodesic metric on $H$.  Hence $d|_H$ is a maximal metric on $H$.  The metric space inclusion $(H,d|_H)\rightarrow(\Diff_+^1(M^1),d)$ is cobounded, since if $M^1=I$ then $H$ is the whole diffeomorphism group, whereas if $M^1=\mathbb{S}^1$, then for every $g\in\Diff_+^1(\mathbb{S}^1)$, there is $h\in H$ and a rotation $r$ such that $rh=g$, whence $d(g,h)=d(r,e)=0$.  Therefore $d$ is quasi-isometric to $d|_H$, so $d$ is large-scale geodesic.  This means $d$ is maximal on $\Diff_+^1(M^1)$, and $\Phi$ becomes a quasi-isometry of $\Diff_+^{1+AC}(\mathbb{S}^1)$ onto $Z$.
\end{proof}

\begin{cor}  $\Diff_+^{1+AC}(I)$ is quasi-isometric to $L^1(I)$ via the mapping $\Phi(f)=\frac{f''}{f'}$.
\end{cor}

\begin{cor}  $\Diff_+^{1+AC}(\mathbb{S}^1)$ is quasi-isometric to $Y=\{H\in L^1(\mathbb{S}^1):\int_{\mathbb{S}^1}F=0\}$ via the mapping $\Phi(f)=\frac{f''}{f'}$.
\end{cor}

\begin{cor}  $\Diff_+^{1+AC}(I)$ and $\Diff_+^{1+AC}(\mathbb{S}^1)$ are quasi-isometric to $L^1(I)$.
\end{cor}

\begin{proof}  The third corollary follows immediately from the previous two if $L^1(I)$ is isomorphic (and hence quasi-isometric) to each of its hyperplanes.  Equivalently, we want to show that $L^1(I)$ is isomorphic to the direct sum $L^1(I)\oplus\mathbb{R}$.  

To see this, we consider the mapping $\varphi:L^1(I)\rightarrow L^1(I)$ defined by setting $\varphi(F)$ equal to the constant $\left(\frac{1}{j}-\frac{1}{j+1}\right)\displaystyle\int_{(1/(j+1),1/j]}F$ on the interval $\left(\frac{1}{j},\frac{1}{j+1}\right]$, for each positive integer $j$.  Then $\varphi$ is a bounded linear projection, and hence $L^1(I)$ is isomorphic to the direct sum $\ker(\varphi)\oplus \varphi(L^1(I))$.  Moreover the image $\varphi(L^1)(I)$ is clearly isomorphic to the sequence space $\ell^1$, by simply mapping $\varphi(F)$ to $(\varphi(F)(1/j))_{j=1}^\infty$.  So $L^1(I)\cong \ker(\varphi)\oplus\ell^1$.  But $\ell^1$ is isomorphic to $\ell_1\oplus\mathbb{R}$ via index shifting, so we get $L^1(I)\cong \ker(\varphi)\oplus\ell_1\oplus\mathbb{R}\cong L^1(I)\oplus\mathbb{R}$.
\end{proof}

\section{Appendix}

As we mentioned in the introduction, there is a simple error in \cite{cohen_2016a}: the maximal metric provided there if right-invariant for $\Diff_+^1(I)$, but is not right-invariant for $\Diff_+^1(\mathbb{S}^1)$.  To correct the mistake, replace the mapping $\phi_1:\Diff_+^k(\mathbb{S}^1)\rightarrow C[0,1]$, $\phi_1(f)=\log f'-\log f'(0)$ given just before \cite{cohen_2016a} Lemma 2.3 with the mapping $\phi_1(f)=\log f'$.  Implementing this change, the continuous pseudometric $d_1$ defined just before \cite{cohen_2016a} Lemma 3.1 becomes $d_1(f,g)=\displaystyle\sup_{x\in\mathbb{S}^1}|\log f'-\log g'|$ (which is the pseudometric we are using in the present article) while the mapping $\Phi_k$ and the metrics $d_k$, $k\geq 2$ are unchanged.

Next one must check that the proofs in \cite{cohen_2016a} are still true for the modified definitions of $\phi_1$ and $d_1$.  By inspection, \cite{cohen_2016a} Lemma 3.2 remains true after a trivial modification of the proof; while the remainder of the theorems labeled 3.3 through 3.7 go through word for word.  Theorem 3.8 of \cite{cohen_2016a} characterizes the coarsely bounded subsets of $\Diff_+^k(\mathbb{S}^1)$.  The proof is still valid when $k\geq 2$, in light of the validity of 3.2--3.8.  However, the proof in the case $k=1$ requires a correction, because the mapping $\phi_1$ is no longer a surjection onto a linear subspace of $C(I)$.  We present a different proof below.

\begin{thm} \label{c1_coarsely_bounded}  A subset $A\subseteq\Diff_+^1(\mathbb{S}^1)$ is coarsely bounded if and only if $\displaystyle\sup_{f\in A}\sup_{x\in \mathbb{S}^1}|\log f'(x)|<\infty$.
\end{thm}

\begin{proof}  Let $A\subseteq\Diff_+^1(M^1)$ be a coarsely bounded set, and let $U$ be the $d_1$-ball about identity of radius $1$.  It means that there is a finite set $F$ and an integer $r$ with $A\subseteq FU^r$.  Since $d$ is right-invariant, $FU^r$ is $d_1$-bounded, and hence so is $A$.  Therefore $\displaystyle\sup_{f\in A}\sup_{x\in \mathbb{S}^1}|\log f'(x)|<\infty$.

Conversely, suppose $\displaystyle\sup_{f\in A}\sup_{x\in \mathbb{S}^1}|\log f'(x)|=M<\infty$.  Let $H$ denote the stabilizer of $0$ in $\Diff_+^1(\mathbb{S}^1)$, so $H$ is a closed subgroup of $\Diff_+^1(\mathbb{S}^1)$, and the restriction of $d_1$ to $H$ is a right-invariant metric on $H$.  Let $K$ denote the compact subgroup of rotations.  Define:

\begin{center} $A^*=\{a^*\in H:\exists t\in\mathbb{S}^1 ~ \exists a\in A ~ ta^*=a\}$.
\end{center}
\vspace{.3cm}

Note that if $f\in A^*$, then we may write $f=tg$ for $t$ a rotation and $g\in A$, and thus $\log f'=\log g'$.  So $\displaystyle\sup_{f\in A}\sup_{x\in \mathbb{S}^1}|\log f'(x)|=M<\infty$.

We will first show that $A^*$ is coarsely bounded.  Let $U\subseteq\Diff_+^1(\mathbb{S}^1)$ be an arbitrary open set.  Then $U\cap H$ is open in $H$, and thus contains a $d_1$-ball in $H$ of radius $\epsilon$; let $V$ denote this ball.  Let $f\in A$ be arbitrary.  Let $N$ be a large integer to be determined later, and consider the functions\\

\begin{center} $f_i(x)=\frac{i}{N}x+(1-\frac{i}{N})f(x)$
\end{center}
\vspace{.3cm}

\noindent where $\id$ denotes the identity diffeomorphism.  Observe that by hypothesis we have $e^{-M}\leq f'\leq e^M$, and therefore\\

\begin{center} $e^{-M} \leq f' \leq f'+\frac{i}{N}(1-f') = f_i'\leq e^M + 1$.
\end{center}
\vspace{.3cm}

Let $K$ be a Lipschitz constant for $\log$ on the interval $[e^{-M},e^M+1]$.  Then we have for each $1\leq i\leq N$, for each $x\in\mathbb{S}^1$,\\

\begin{align*}
|\log f_i'(x)-\log f_{i-1}'(x)| &\leq K|f_i'(x)-f_{i-1}'(x)|\\
&\leq K|\frac{1}{N}(1-f')|\\
&\leq \frac{K(1-e^{-M}}{N}.
\end{align*}
\vspace{.3cm}

We now choose $N$ to be so large that the last line above is $<\epsilon$.  Set $v_i=f_if_{i-1}^{-1}$, so $d(v_i,e)=d(f_i,f_{i-1})<\epsilon$ and $v_i\in V$.  We have $f=v_1v_2...v_N$, so $f\in V^N\subseteq U^N$.  Since $f\in A^*$ was arbitrary, this shows $A^*\subseteq U^N$, and since $U$ was an arbitrary open set, this shows $A^*$ is coarsely bounded.  Since $K$ is compact, $K$ is coarsely bounded, and therefore the product $KA^*$ is coarsely bounded.  Since $A\subseteq KA^*$, $A$ is coarsely bounded, completing the proof.
\end{proof}

The quasi-isometry type of $\Diff_+^1(\mathbb{S}^1)$ is computed in \cite{cohen_2016a} Theorem 4.3; however the proof relies on the erroneous metric.  Here we give the result with a corrected proof.

\begin{thm}  The pseudometric $d(f,g)=d_1(f,g)=\displaystyle\sup_{x\in\mathbb{S}^1}|\log f'-\log g'|$ defined on $\Diff_+^1(\mathbb{S}^1)$ is maximal, and $\Diff_+^1(\mathbb{S}^1)$ is quasi-isometric to $Z=\{f\in C(I):f(0)=f(1)=0\}$ via the mapping $f\mapsto \log f'-\log f'(0)$.
\end{thm}

\begin{proof}  Again let $H$ denote the stabilizer of $0$ in $\Diff_+^1(\mathbb{S}^1)$.  By Theorem \ref{c1_coarsely_bounded}, $d$ is coarsely proper as well as right-invariant.   Thus its restriction to $H$, which we denote $d|_H$, has the same properties.

The mapping $f\mapsto \log f'-\log f'(0)$ is a bijection from $H$ onto $Z$ (see \cite{cohen_2016a}), and thus we induce a metric $\sigma$ on $H$ defined by $\sigma(f,g)=\|\log f'-\log f'(0)-(\log g'-\log g'(0))\|$.  Being isometric to the norm metric on $Z$, $\sigma$ is a geodesic metric.  We note that for any $f\in H$, $\log f'(0)\leq \|\log f'\|$; from this it is easy to compute that $\sigma(f,g)\leq 2d|_H(f,g)$.  On the other hand, by the mean value theorem, there exists $x\in\mathbb{S}^1$ so that $\log f'(x)=0$, and therefore $\log f'(0)=-(\log f'(x)-\log f'(0))\leq\|\log f'-\log f'(0)\|$.  From this we deduce $d_|H(f,g)\leq 2\sigma(f,g)$.  So $d|_H$ and $\sigma$ are bi-Lipschitz equivalent.  Since $\sigma$ is geodesic, it follows that $d|_H$ is large-scale geodesic, and hence $d|_H$ is a maximal metric on $H$ by \cite{rosendal_2018a} Proposition 2.52.  Moreover, $H$ is quasi-isometric to $Z$.

Lastly, note that the metric space inclusion $(H,d|_H)\rightarrow(\Diff_+^1(\mathbb{S}^1),d)$ is cobounded, since for every $g\in\Diff_+^1(\mathbb{S}^1)$, there is $h\in H$ and a rotation $r$ such that $rh=g$, and $d(g,h)=d(r,e)=0$.  Thus $d$ is quasi-isometric to $d|_H$; hence $d$ is large-scale geodesic and maximal on $\Diff_+^1(\mathbb{S}^1)$.  The conclusions of the theorem follow immediately.
\end{proof}

Since $C(I)$ is isomorphic to its hyperplanes, we recover the following.

\begin{cor}  $\Diff_+^1(\mathbb{S}^1)$ is quasi-isometric to $C(I)$.
\end{cor}

\end{document}